\documentclass{amsart}
\usepackage{amssymb,amsxtra}
\usepackage{graphicx}
\usepackage{amscd}
\usepackage{amsmath}
\usepackage{amsfonts}
\usepackage{amssymb}
\theoremstyle{plain}

\newtheorem{theorem}{Theorem}[section]
\newtheorem{proposition}{Proposition}[section]
\newtheorem{lemma}{Lemma}[section]
\newtheorem{corollary}{Corollary}[section]\theoremstyle{definition}
\newtheorem{definition}{Definition}[section]
\newtheorem{example}{Example}[section]

\newtheorem{remark}{Remark}[section]
\theoremstyle{example}

\title[Entropy of shifts on topological graph $C^*$-algebras]{Entropy of shifts on topological graph $C^*$-algebras}
\author{Valentin Deaconu}
\address{Department of Mathematics\\ University of Nevada\\ Reno NV
89557-0084, USA}
\email[Valentin Deaconu]{vdeaconu@unr.edu}

\thanks
{Research supported by a UNR JFR Grant}

\date{31 December 2008}
\begin{document}

\begin{abstract}
We give entropy estimates   for  two canonical non commutative shifts on $C^*$-algebras  associated to some topological graphs $E=(E^0,E^1,s,r)$, defined using a basis of the corresponding Hilbert bimodule $H(E)$. We compare their entropies with the growth entropies associated directly to the topological graph.  We illustrate with some examples of topological graphs considered by Katsura, where the vertex  and the edge spaces are a union of unit circles and more detailed computations can be done. \end{abstract}

\maketitle

\section{Introduction}

The topological entropy for automorphisms of nuclear $C^*$-algebras was introduced by Voiculescu in \cite{Vo} and extended by Brown to exact $C^*$-algebras in \cite{Br}. It was computed for many examples of automorphisms, endomorphisms, and completely positive (cp) maps by several authors, using often the fact that the restriction to a commutative $C^*$-algebra is a  map  for which the entropy is known. For a comprehensive treatment of various techniques of computation of this entropy, we refer to \cite{NS}.

Recall that the map \[ \Phi:{\mathcal O}_n\to {\mathcal O}_n, \;\;\Phi(c)=\sum_{i=1}^nS_icS_i^*,\] where ${\mathcal O}_n$ is the Cuntz algebra with generators $S_1,...,S_n$, is an endomorphism and has topological entropy $\log n$, see \cite{Ch}. The map $\Phi$ invariates the AF-core ${\mathcal F}_n\cong UHF(n^\infty)$ generated by monomials $S_\alpha S_\beta^*$ with $|\alpha|=|\beta|$  and the abelian algebra ${\mathcal D}_n$ generated by monomials $S_\alpha S_\alpha^*$.  Recall that for a word $\alpha=\alpha_1\cdots \alpha_k\in\{1,2,...,n\}^k$, $k\in{\mathbb N}$, we define $S_\alpha:=S_{\alpha_1}\cdots S_{\alpha_k}$ and $|\alpha|=k$. It is known that ${\mathcal D}_n$ is  isomorphic to $C(X)$, where $X=\{1,2,...,n\}^{\mathbb N}$, and $\Phi$ is called a non commutative shift, since $\Phi\mid_{{\mathcal D}_n}$ is conjugated to the map $\tilde{\sigma}:C(X)\to C(X), \;\;\tilde{\sigma}(f)=f\circ\sigma$,  where  $\sigma$ is the unilateral shift $\sigma:X\to X,\;\; \sigma(x_0x_1x_2...)=x_1x_2...$.

In the  case of the Cuntz-Krieger algebra ${\mathcal O}_\Lambda$, where $\Lambda$ is an incidence matrix, the corresponding map $\Phi$ is no longer multiplicative, but it is a unital completely positive map. Boca and Goldstein (see \cite{BG}) proved that the topological entropy of $\Phi$ is $\log \rho(\Lambda)$, where $\rho(\Lambda)$ is the spectral radius of $\Lambda$. This coincides with  the classical topological entropy of the underlying Markov shift $(X_\Lambda,\sigma)$. 
Similar results are obtained for graph $C^*$-algebras by Jeong and Park, see \cite{JP1, JP2, JP3}, and for higher-rank graph $C^*$-algebras by Skalski and Zacharias, see \cite{SZ}.
The Boca-Goldstein technique was also used by Kerr and Pinzari to analyze the noncommutative pressure and the variational principle in Cuntz-Krieger type $C^*$-algebras, see \cite{KP}.

Let $H$ be a full Hilbert bimodule over a $C^*$-algebra $A$ with a basis $\xi_1,...,\xi_n$, in the sense that for all $\xi\in H$,
\[\xi=\sum_{i=1}^n\xi_i\langle \xi_i,\xi\rangle.\]
The Cuntz-Pimsner algebra ${\mathcal O}_H$ is generated by $A$ and $S_i=S_{\xi_i}, 1\le i\le n$ with relations
\[\sum_{i=1}^nS_iS_i^*=1,\; S_i^*S_j=\langle \xi_i,\xi_j\rangle,\; a\cdot S_j=\sum_{i=1}^nS_i\langle \xi_i,a\cdot \xi_j\rangle, \;a\in A,\]
see \cite{KPW}. We consider the ucp map
\[\Phi:{\mathcal O}_H\to{\mathcal O}_H, \;\Phi(c)=\sum_{i=1}^nS_icS_i^*,\] and call it a non commutative shift. Notice that $\Phi$ leaves invariant the core ${\mathcal F}_H$, generated by monomials $S_\alpha aS_\beta^*$ with $a\in A$ and $|\alpha|=|\beta|$, and  also invariates the subalgebra ${\mathcal C}_H$ generated by monomials $S_\alpha aS_\alpha^*$. 
If $H={\mathbb C}^n$, we recover the canonical endomorphism considered by Choda, since  ${\mathcal O}_H={\mathcal O}_n$, the Cuntz algebra generated by $n$ isometries $S_1,...,S_n$. It is our goal to study the map $\Phi$ and its entropy for more general finitely generated Hilbert bimodules.

In section 2 we recall  the definition of the $C^*$-algebras $C^*(E)$ and ${\mathcal F}_E$ associated to a topological graph $E$. In section 3  we give more details about the structure of these $C^*$-algebras, in the presence of a basis $\{\xi_i\}_{1\le i\le n}$ of the Hilbert bimodule. We also define the ucp map $\Phi$ associated with such a basis. In the context of topological graph $C^*$-algebras there is another candidate for the notion of non commutative shift, denoted by $\Psi$. This is defined only on the core algebra ${\mathcal F}_E$, using the embeddings  ${\mathcal K}(H^{\otimes k})\to {\mathcal K}(H^{\otimes k+1})$ for some particular topological graphs. For these  graphs, we can associate an  \' etale groupoid as in \cite{De2}, and the map $\Psi$ restricted to the  diagonal ${\mathcal D}_E=C(E^\infty)$ coincides with the shift map on the unit space of the groupoid identified with the space of infinite paths $E^\infty$.  In section 4, we define the loop and block entropies for  topological graphs $E$ and  compute them for several examples. In section 5 we study  the entropy of the non commutative shifts $\Phi$ and $\Psi$, and the relationship with these growth entropies.  The main result computes the entropy of $\Phi$ in terms of the spectral radius of an incidence matrix. The entropy of $\Psi$ is bounded below by the loop entropy. Compared to the situation of discrete graphs or higher rank graphs, new phenomena occur, since $\Phi\mid_{{\mathcal F}_E}$ and $\Psi$ are different and they may have different entropies. We illustrate this with several examples in section 6. We also recover earlier computations from \cite{De1}.

Our results are similar with results of Pinzari, Watatani and Yonetani in \cite{PWY}. 
In the final stages of writing this paper, we also learned that  Yamashita, in his  preprint \cite {Y},  obtained
similar entropy computations for a particular class of toplogical graphs, called circle correspondences.

{\bf Acknowledgements}. This research was supported by a University of Nevada JFR Grant.

\section{Topological graphs and their $C^*$-algebras}

Topological graphs were studied  in \cite{De2} under the name continuous graphs, and generalized by several authors. We are using the terminology and many facts from  papers of Katsura, see \cite{Ka1}. 
Let $E=(E^0,E^1,s,r)$ be a  topological graph. Recall that $E^0,E^1$ are locally compact spaces, and $s,r:E^1\to E^0$ are continuous maps such that $s$ is a local homeomorphism. We think of points in $E^0$ as vertices, and of elements  $e\in E^1$ as edges from $s(e)$ to $r(e)$. Several examples of topological graphs will be considered in section 4.

\begin{definition}The  $C^*$-algebra of a graph $E$, denoted $C^*(E)$ is defined to be the Cuntz-Pimsner algebra ${\mathcal O}_H$, where the Hilbert bimodule $H=H(E)$ over the $C^*$-algebra $A=C_0(E^0)$ is obtained as the completion of $C_c(E^1)$ using the inner product 
\[\langle \xi,\eta\rangle(v)=\sum_{s(e)=v}\overline{\xi(e)}\eta(e),\; \xi,\eta\in C_c(E^1)\]
and the multiplications
\[(\xi\cdot f)(e)=\xi(e)f(s(e)),\; (f\cdot\xi)(e)=f(r(e))\xi(e).\] 
The core algebra ${\mathcal F}_E$ is the fixed point algebra under the gauge action, see below.
\end{definition}

Recall that a Hilbert bimodule over a $C^*$-algebra $A$ (sometimes called a C*-correspondence from $A$ to $A$) is a Hilbert $A$-module $H$ with a left action of $A$ given by a homomorphism $\varphi:A\rightarrow {\mathcal L}(H)$, where ${\mathcal L}(H)$ denotes the $C^*$-algebra of all adjointable operators on $H$. A Hilbert bimodule  is called  full if the inner products generate $A$.   For $n\ge 0$ we denote by $H^{\otimes n}$ the Hilbert bimodule obtained by taking the tensor product of $n$ copies of $H$, balanced over $A$ (for $n=0, H^{\otimes 0}=A$). For $n=2$, the inner product is given by \[\langle\xi\otimes\eta,\xi'\otimes\eta'\rangle=\langle\eta,\varphi(\langle\xi,\xi'\rangle)\eta'\rangle,\] and it is inductively defined for general $n$.

A {\em Toeplitz representation} of a Hilbert bimodule $H$ over $A$ in a C*-algebra $C$ is a pair $(\tau,\pi)$ with $\tau:H\rightarrow C$ a linear map and $\pi:A\rightarrow C$ a *-homomorphism, such that 
\[\tau(\xi a)=\tau(\xi)\pi(a), \;\tau(\xi)^*\tau(\eta)=\pi(\langle \xi,\eta\rangle), \;\tau(\varphi(a)\xi)=\pi(a)\tau(\xi).\]
The $C^*$-algebra generated by the images of $\pi$ and $\tau$ in $C$ is denoted by $C^*(\tau,\pi)$.
The corresponding universal C*-algebra is called the Toeplitz algebra of $H$, denoted by ${\mathcal T}_H$. If $H$ is full, then ${\mathcal T}_H$ is generated by elements $\tau^n(\xi)\tau^m(\eta)^*, m,n\ge 0$, where $\tau^0=\pi$  and for $n\geq 1, \tau^n(\xi_1\otimes ...\otimes\xi_n)=\tau(\xi_1)...\tau(\xi_n)$ is the extension of $\tau$ to $H^{\otimes n}$. Note that $A$ is isomorphic to a subalgebra of ${\mathcal T}_H$.

The rank one operators $\theta_{\xi,\eta}$  given by $\theta_{\xi,\eta}(\zeta)=\xi\langle\eta,\zeta\rangle$ generate an essential ideal of ${\mathcal L}(H)$, denoted ${\mathcal K}(H)$. It is known that ${\mathcal K}(H)\cong H\otimes H^*$, where $H^*$ is the dual of $H$, regarded as a left $A$-module. The rank one operator $\theta_{\xi,\eta}\in{\mathcal K}(H)$ is identified with $\xi\otimes\eta^*\in H\otimes H^*$. The map $\tau$ defines  a homomorphism $\psi:{\mathcal K}(H)\rightarrow C$ such that $\psi(\xi\otimes\eta^*)=\tau(\xi)\tau(\eta)^*$.  
A representation $(\tau,\pi)$ is {\em Cuntz-Pimsner covariant} if $\pi(a)=\psi(\varphi(a))$ for all $a$ in the ideal \[I_H=\varphi^{-1}({\mathcal K}(H))\cap (\ker \varphi)^\perp.\] 
The Cuntz-Pimsner algebra ${\mathcal O}_H$ is universal with respect to the covariant representations, and it is a quotient of ${\mathcal T}_H$.
The universal properties allow us to define a gauge action of ${\mathbb T}$ on ${\mathcal T}_H$ and ${\mathcal O}_H$ such that
\[z\cdot(\tau^n(\xi)\tau^m(\eta)^*)=z^{n-m}\tau^n(\xi)\tau^m(\eta)^*, z\in{\mathbb T}.\]
 The {\em core} $C^*$-algebra ${\mathcal F}_H$ is the fixed point algebra ${\mathcal O}_H^{\mathbb T}$, and it is generated by the union of the algebras ${\mathcal K}(H^{\otimes k})$, $k\ge 0$. For more details about the algebras ${\mathcal F}_H, {\mathcal T}_H$, and ${\mathcal O}_H$ we refer to the papers of Pimsner (\cite{Pi}) and Katsura (\cite{Ka1}).
 
\begin{remark} Notice that, since Hilbert bimodules over commutative $C^*$-algebras are associated to continuous fields of Hilbert spaces, it follows that for $H=H(E)$, where $E=(E^0,E^1,s,r)$ is a topological graph with $s$ surjective,  ${\mathcal K}(H)$ is a continuous trace $C^*$-algebra over $E^0$. The elements of ${\mathcal K}(H)$ can be thought as compact operator valued continuous functions on $E^0$ which vanish at infinity, and the elements of ${\mathcal L}(H)$ can be considered as bounded operator valued continuous functions on $E^0$.
 \end{remark}
 
 \section{Finitely generated bimodules and non commutative shifts}
 
\begin{definition}The full Hilbert bimodule $H$ over $A$ is  finitely generated
if it has a basis $\{\xi_i\}_{1\le i\le n}$, in the sense that for all $\xi\in H$,
\[\xi=\sum_{i=1}^n\xi_i\langle \xi_i,\xi\rangle.\]
\end{definition}

It is easy to check that, in this case,   $H^{\otimes k}$ has basis $\xi_{i_1}\otimes\cdots \otimes \xi_{i_k}$, where $i_1,...,i_k\in\{1,2,...,n\}$. 
Denote by $S_i$ the image of $\xi_i$ in the Cuntz-Pimsner algebra ${\mathcal O}_H$. Then ${\mathcal O}_H$ is generated  by $A$ and $S_i, 1\le i\le n$ with relations
\[\sum_{i=1}^nS_iS_i^*=1,\; S_i^*S_j=\langle \xi_i,\xi_j\rangle,\; a\cdot S_j=\sum_{i=1}^nS_i\langle \xi_i,a\cdot \xi_j\rangle, \;a\in A,\]
see \cite{KPW}.
For a word $\alpha=\alpha_1\cdots \alpha_k\in\{1,2,...,n\}^k$ of length $|\alpha|=k$, let $S_\alpha=S_{\alpha_1}\cdots S_{\alpha_k}$.

\begin{lemma} Assume that the full Hilbert bimodule $H$ over $A$ has a basis $\{\xi_i\}_{1\le i\le n}$ as above. Then ${\mathcal K}(H)={\mathcal L}(H)$, and the Cuntz-Pimsner algebra ${\mathcal O}_H$ is the closed linear span of monomials $S_\alpha aS_\beta^*$, where $\alpha, \beta$ are arbitrary words and $a\in A$. The core algebra ${\mathcal F}_H$ is the closed linear span of monomials $S_\alpha aS_\beta^*$ with $|\alpha|=|\beta|$.
\end{lemma}
\begin{proof}
Since $\sum_{i=1}^n S_iS_i^*=1$, it is clear that ${\mathcal K}(H)$ is unital, and therefore ${\mathcal K}(H)={\mathcal L}(H)$. It suffices to show that the product of two  monomials as above is a sum of monomials of the same form.
Consider the product $S_\alpha aS_\beta^*S_\gamma bS_\delta^*$ for some words $\alpha,\beta,\gamma,\delta$ and $a,b\in A$. Using the relations $S_i^*S_j=\langle \xi_i,\xi_j\rangle\in A$ and $a\cdot S_j=\sum_{i=1}^nS_i\langle \xi_i,a\cdot \xi_j\rangle$ repeatedly, for $|\beta|=|\gamma|$ we get that $S_\beta^*S_\gamma$ belongs to $A$; for $|\beta|>|\gamma|$ we get $S_\beta^*S\gamma=S_{\beta'}^*c$ with $|\beta'|=|\beta|-|\gamma|$ and $c\in A$; and for $|\beta|<|\gamma|$ we get $S_\beta^*S_\gamma=dS_{\gamma'}$ with $|\gamma'|=|\gamma|-|\beta|$ and $d\in A$. In the first case we are done. In the second and the third case, we use again the relation $a\cdot S_j=\sum_{i=1}^nS_i\langle \xi_i,a\cdot \xi_j\rangle$ repeatedly to write
\[S_{\beta'}^*c=(c^*S_{\beta'})^*=(\sum_j S_{\mu^j}c_j)^*=\sum_jc_j^*S_{\mu^j}^*\]
and respectively
\[dS_{\gamma'}=\sum_iS_{\nu^i}d_i,\]
where $\mu^j, \nu^i$ are words, and $c_j, d_i\in A$.
We get \[S_\alpha aS_\beta^*S_\gamma bS_\delta^*=\sum_jS_\alpha ac_j^*S_{\mu^j}^*bS_\delta^*\]
and
\[S_\alpha aS_\beta^*S_\gamma bS_\delta^*=\sum_iS_\alpha aS_{\nu^i}d_ibS_\delta^*.\]
We can use again the above relations to write $S_{\mu^j}^*b$ as $\sum_kb_{jk}S_{\mu^{jk}}^*$ with $b_{jk}\in A$ and
$aS_{\nu^i}=\sum_lS_{\nu^{il}}a_{il}$ with $a_{il}\in A$ to get the desired result. For the core algebra we can use the same method, keeping track of the lengths of the involved words.
\end{proof}

\begin{corollary}
If in addition  $\langle\xi_i,\xi_j\rangle=0$ for all $1\le i\neq j\le n$,  then $S_i$ are partial isometries. Set $Q_i:=S_i^*S_i$ and  $P_i:=S_iS_i^*$ for $i=1,...,n$. Suppose that $Q_i=\sum_{j=1}^n\Lambda(i,j)P_j$ for some incidence matrix $\Lambda$. Then the non-zero elements $S_\mu P_iS_\nu^*$ for $|\mu|=|\nu|=k$ and $i=1,...,n$ form a system of matrix units generating a finite dimensional $C^*$-algebra $F_k$.
 If, moreover, $A$ is commutative and unital with $A\subset{\mathcal K}(H)$, then there is an isomorphism $F_k\otimes A\cong{\mathcal K}(H^{\otimes k})\cong {\mathcal L}(H^{\otimes k})$,
and the Cuntz-Pimsner algebra ${\mathcal O}_H$  contains a copy of the Cuntz-Krieger algebra ${\mathcal O}_\Lambda$. The core ${\mathcal F}_H$ is isomorphic to the inductive limit $\varinjlim{\mathcal K}(H^{\otimes k})\cong\varinjlim F_k\otimes A$, where the embeddings are determined by the left action of $A$, and it  contains a copy of the stationary AF-algebra determined by $\Lambda$.   There is  a commutative subalgebra ${\mathcal C}_H\subset{\mathcal F}_H$, generated by monomials $S_\alpha aS_\alpha^*$ with $a\in A$, which contains o copy of $C(X_\Lambda)$, where 
 $X_\Lambda=\{(x_k)\in\{1,2,...,n\}^{\mathbb N}:\Lambda(x_k,x_{k+1})=1\}.$
\end{corollary}
\begin{proof}
Since $\langle\xi_i,\xi_j\rangle=0$ for $i\neq j$, we get $\xi_i=\xi_i\langle\xi_i,\xi_i\rangle$, which implies $S_i=S_iS_i^*S_i$ for all $i$.
It follows that each $S_i$ is a partial isometry, and therefore $P_i, Q_i$ are projections. Under our assumptions, it is clear that $S_i$ generate a copy of the Cuntz-Krieger algebra ${\mathcal O}_\Lambda$. The isomorphism $F_k\otimes A\cong{\mathcal K}(H^{\otimes k})$ is given by the map
$S_\mu P_iS_\nu^*\otimes a\mapsto S_\mu aP_iS_\nu^*$. The rest follows from the previous lemma.
\end{proof}

\begin{definition} The canonical non commutative shift associated to a basis $\{\xi_i\}_{1\le i\le n}$ of the Hilbert bimodule $H$ over $A$ is the ucp map
\[\Phi:{\mathcal O}_H\to{\mathcal O}_H, \;\Phi(c)=\sum_{i=1}^nS_icS_i^*.\] 
 \end{definition}
\begin{remark}
Notice that $\Phi$ leaves invariant the core algebra ${\mathcal F}_H$, generated by monomials $S_\alpha aS_\beta^*$ with $a\in A$ and $|\alpha|=|\beta|$, and  also invariates the subalgebra ${\mathcal C}_H$ generated by monomials $S_\alpha aS_\alpha^*$. In the hypotheses of the above corollary, the map $\Phi$  also leaves invariant the subalgebra $C(X_\Lambda)$, and its restriction is conjugated to the map induced by the Markov shift $\sigma_\Lambda:X_\Lambda \to X_\Lambda, \sigma_\Lambda(x_0x_1x_2\cdots)=x_1x_2\cdots$.
\end{remark}

In our examples of topological graphs, $E^0$ and $E^1$ are compact and $s, r$ are surjective local homeomorphisms. In that case, it follows that the left action of $A=C(E^0)$ on  $H(E)$ is by compact operators. Moreover,  the Hilbert module $H=H(E)$ will have a basis $\xi_1,\xi_2,...,\xi_n$ satisfying
 \[\langle\xi_i,\xi_j\rangle=\delta_{ij}Q_i,\]
where $Q_i\in A$ are projections. It follows from the corollary that the $C^*$-algebra $C^*(E)={\mathcal O}_H$ is generated by $A=C(E^0)$ and $n$ partial isometries $S_1, S_2,...,S_n$ with orthogonal ranges satisfying some commutation relations determined by the range map $r$. 

For these topological graphs, there is also a groupoid approach for $C^*(E)$. For $k\ge 2$, let's define $E^k$ to be the space of paths of length $k$ in the topological graph and $E^\infty=\varprojlim E^k$ to be the space of infinite paths. For $k\ge 0$, we define the equivalence relation $R_k$ on $E^k$ by
\[R_k=\{(e,f)\in E^k\times E^k: s(e)=s(f)\},\]
with the induced product topology. Note that $R_0$ is just the diagonal in $E^0\times E^0$. The set
 \[\Gamma=\Gamma(E)=\{(x, p-q, y)\in E^\infty\times{\mathbb Z}\times E^\infty:\sigma^p(x)=\sigma^q(y)\},\]
with natural operations $(x, k,y)(y,l,z)=(x,k+l,z), (x,k,y)^{-1}=(y,-k,x)$ and appropriate topology, becomes an amenable \' etale groupoid,  and its $C^*$-algebra is isomorphic to $C^*(E)$, see \cite{De2}.
 
\begin{definition} The map $\varphi:C(E^0)\to C^*(R_1)$ induces  embeddings $\Psi_k:C^*(R_k)\to C^*(R_{k+1})$ for each $k\ge 0$, which define another non commutative shift $\Psi: {\mathcal F}_E\to {\mathcal F}_E$. This is a $*$-homomorphism.
\end{definition}
 \begin{remark}
 Recall that we have  isomorphisms $C^*(R_k)\cong{\mathcal K}(H^{\otimes k})$.
  There is a natural diagonal isomorphic to $C(E^\infty)$ inside $\varinjlim C^*(R_k)\cong{\mathcal F}_E$,
and the restriction $\Psi\mid_{C(E^\infty)}$  is conjugated to the map defined by the shift  $\sigma:E^\infty\to E^\infty$,  $\sigma(e_1e_2e_3\cdots)=e_2e_3\cdots$, which is a surjective local homeomorphism. 
 \end{remark}
 
 \section{Growth entropies of topological graphs}

Recall that a {\em path} of length $k$ in a topological graph $E=(E^0,E^1,s,r)$ is a concatenation $e_1e_2\cdots e_n$ of edges in $E^1$ such that $s(e_i)=r(e_{i+1})$ for all $i$.  The maps $s$ and $r$ extend naturally to $E^k$. A vertex $v\in E^0$ is viewed as a path of length $0$. 
For each vertex $v\in E^0$ and $k\ge 1$ let
\[E^k_s(v):=\{e\in E^k\mid s(e)=v\},\; E^k_r(v):=\{e\in E^k\mid r(e)=v\}\]
and 
\[E^k_{\ell}(v):=\{e\in E^k\mid s(e)=r(e)=v\},\]
the set of loops of length $k$ based at $v$. We will be mostly interested in locally finite topological graphs in the sense that each vertex emits and receives a finite number of edges. In that case, the sets $E^k_s(v), E^k_r(v)$ and $E^k_{\ell}(v)$ are finite for each $k$ and $v$. If both maps $s$ and $r$ are onto, then $E^k_s(v), E^k_r(v)$ are nonempty for all $k$ and $v$, but $E^k_{\ell}(v)$ may be empty. Define
\[E^k_\ell:=\bigcup_{v\in E^0}E^k_{\ell}(v).\]

\begin{definition} Assume $E$ is a locally finite topological graph. The loop entropy $h_{\ell}(E)$ and the block entropy $h_b(E)$ of $E$ are defined by
\[h_{\ell}(E)=\limsup_{k\to\infty}\frac{1}{k}\log|E^k_{\ell}|,\;h_b(E)=\sup_{v\in E^0}\limsup_{k\to\infty}\frac{1}{k}\log|E^k_r(v)|,\]
where $|L|$ denotes the number of elements in $L$.
\end{definition}

\begin{remark} 
For a topological graph $E$ where $r$ is also a local homeomorphism, we can interchange $s$ and $r$ to get a new topological graph, denoted $E^t$ and called the  transposed graph. We have $h_{\ell}(E)=h_{\ell}(E^t)$ but $h_b(E)\neq h_b(E^t)$ in general.  It is known that the loop entropy and the block entropy make sense for discrete graphs, and that these entropies can be computed in terms of the incidence matrix. For example, if the irreducible finite graph $E$ is given by a matrix $\Lambda_E$, then it is known that $h_{\ell}(E)=h_b(E)=\log \rho(\Lambda_E)$, where $\rho(\Lambda_E)$ denotes the spectral radius. Since $\rho(\Lambda_E)=\rho(\Lambda_E^t)$, for finite graphs we have $h_b(E)=h_b(E^t)$. This equality fails for infinite graphs, as was shown by Jeong and Park in Example 3.3 of \cite{JP2}.
 For topological graphs,  see the next examples.
\end{remark}

\begin{example} Let $X$ be a compact metric space, and $\sigma:X\to X$ a continuous surjective map. Define the topological graph $E=E(X,\sigma)=(E^0,E^1,s,r)$, where $E^0=E^1=X,\; s=id$ and $r=\sigma$. For $k\ge 2$, the space
\[E^k=\{(x_1,...,x_k)\in X\times ...\times X\mid x_j=\sigma(x_{j+1}), j=1,...,k-1\}\]
is homeomorphic to $X$. For a fixed $v\in E^0=X$ we can identify $E^k_s(v)$ with $\{v\}$,  $E^k_r(v)$ with $\sigma^{-k}(v)$ and $E^k_{\ell}$ with $\{v\in X\mid \sigma^k(v)=v\}$. Hence 
\[h_{\ell}(E)= \limsup_{k\to\infty}\frac{1}{k}\log|\{x\in X\mid\sigma^k(x)=x\}|,\; h_b(E)=\sup_{v\in X}\limsup_{k\to \infty}\frac{1}{k}\log |\sigma^{-k}(v)|.\]
\noindent Recall that for an expansive map $\sigma:X\to X$ of a metric space $X$, \[h_{top}(\sigma)\ge \limsup_{k\to\infty}\frac{1}{k}\log|\{x\in X\mid\sigma^k(x)=x\}|,\] see theorem 8.16 in \cite{Wa}, which gives the inequality
\[h_{top}(\sigma)\ge h_{\ell}(E(X,\sigma)).\] 
\end{example}
\begin{corollary} A topological graph $E$ with $E^1$ a compact metric space determines a local homeomorphism \[\sigma=\sigma_E:E^\infty\to E^\infty, \sigma(e_1e_2e_3...)=e_2e_3...,\] where \[E^\infty=\{e_1e_2...\in (E^1)^{\mathbb N} \mid s(e_j)=r(e_{j+1})\}\] is the space of infinite paths. Then the set of fixed points of $\sigma$ coincides with the union of loops $\bigcup_{k\ge 1}E^k_{\ell}$. If $\sigma$ is expansive, it follows that 
\[h_{top}(\sigma)\ge h_{\ell}(E).\]
\end{corollary}
\begin{example} Consider the topological graphs $E_{p,q}$ from Example A.6  of \cite{Ka3}. Recall that the vertex and edge spaces are copies of the unit circle ${\mathbb T}$,  the source and range maps are $\;s(z)=z^p$ and $r(z)=z^q$, where  $p, q$ are integers with $p\ge 1$ and $q\neq 0$. Let's assume  that $(p,q)=1$. Then it is easy to see that
\[|E^k_s(v)|=p^k,\;\;|E^k_r(v)|=|q|^k,\;\; |E^k_{\ell}|=|p^k-|q|^k|.\]
It follows that for this topological graph, $h_b(E)=\log |q|,\; h_b(E^t)=\log p$, and $h_{\ell}(E)=\log\max\{ p,|q|\}$.
The space of infinite paths $E_{p,q}^{\infty}$ is homeomorphic to 
\[\Sigma(p,q)=\{(z_0,z_1,z_2,...)\in {\mathbb T}^{\mathbb N}\mid z_k^p=z_{k+1}^q\},\] 
which is a $1$-dimensional solenoid.
 In particular, the shift $\sigma:\Sigma(p,q)\to\Sigma(p,q),\; \sigma(z_0,z_1,z_2,...)=(z_1,z_2,...)$ has entropy $\ge\log\max\{p, |q|\}$.  For a more general result about the entropy of automorphisms of solenoids, see \cite{Kw} and \cite{LW}.
\end{example}

\begin{example} Consider the topological graph $E=(E^0,E^1,s,r)$ where $E^0={\mathbb T}, E^1={\mathbb T}\times\{1,2\}, s(z,1)=z^2, r(z,1)=z, s(z,2)=z, r(z,2)=z^3$. The space of infinite paths is the generalized solenoid 
\[E^\infty=\{(z_k,a_k)_{k\ge 0}\in ({\mathbb T}\times \{1,2\})^{\mathbb N}\mid a_k=1 \Rightarrow \; z_{k+1}^2=z_k, \;a_k=2 \Rightarrow \;z_{k+1}=z_k^3\}.\]
Counting all the possible words
$(z_1,a_1)(z_2,a_2)\cdots(z_m,a_m)$ for a fixed $v=r(z_1,a_1)\in{\mathbb T}$, we get 
$|E^m_r(v)|=4^m.$ 
It follows that $h_b(E)=\log 4$.
Counting all the possible words
$(z_1,a_1)(z_2,a_2)\cdots(z_m,a_m)$ for a fixed $v=s(z_m,a_m)\in{\mathbb T}$, we get 
$|E_s^m(v)|=3^m,$ 
hence $h_b(E^t)=\log 3$. Counting all the loops $(z_1,a_1)(z_2,a_2)\cdots(z_m,a_m)$, we get \[|E^m_\ell|=\sum_{k=0}^m\left(\begin{array}{c}m\\k\end{array}\right)|2^{m-k}-3^k|.\] It follows that $h_\ell(E)=\log 4$,
since
\[4^m-3^m=|\sum_{k=0}^m\left(\begin{array}{c}m\\k\end{array}\right)(2^{m-k}-3^k)|\le\sum_{k=0}^m\left(\begin{array}{c}m\\k\end{array}\right)|2^{m-k}-3^k|\le\sum_{k=0}^m\left(\begin{array}{c}m\\k\end{array}\right)3^k=4^m.\]
\end{example}
\begin{example} More general, for a discrete graph $E=(E^0,E^1,s,r)$ with no sinks, Katsura considers in \cite{Ka3} the topological graph $E\times_{p,q}{\mathbb T}$ with vertex space $E^0\times{\mathbb T}$, edge space $E^1\times{\mathbb T}$, source map $s(e,z)=(s(e),z^{p(e)})$ and range map $r(e,z)=(r(e),z^{q(e)})$, where $p:E^1\to \{1,2,3,...\}$ and $q:E^1\to{\mathbb Z}$ are  two maps. 
The infinite path space is the generalized solenoid
\[(E\times_{p,q}{\mathbb T})^\infty=\{(e_1,e_2,...; z_1,z_2,...)\in E^\infty\times{\mathbb T}^\infty\mid z_k^{p(e_k)}=z_{k+1}^{q(e_{k+1})}\; \text{for}\;\; k=1,2,...\}.\]
Define the matrices $P,Q$, where
\[P(v,w)=\sum_{e\in s^{-1}(v)\cap r^{-1}(w)}p(e), \;\;Q(v,w)=\sum_{e\in s^{-1}(v)\cap r^{-1}(w)}q(e)\] for $v,w\in E^0$. 

We believe that, for $E$ finite, under certain conditions on the maps $p$ and $q$, we have
\[h_b(E\times_{p,q}{\mathbb T})=\log\rho(Q),\;\; h_b((E\times_{p,q}{\mathbb T})^t)=\log\rho(P),\;\;h_\ell(E\times_{p,q}{\mathbb T})=\log\max\{\rho(P),\rho(Q)\},\]
but we were unable to prove it at this time. 
\end{example}

\section {Entropy of non commutative shifts}

We recall the definition and a few useful facts concerning the topological entropy of completely positive maps. The reader may consult \cite{Br} or \cite{NS} for an extensive treatment.
Let $A$ be an exact $C^*$-algebra, $\pi:A\to {\mathcal L}({\mathcal H})$ a faithful *-representation on a Hilbert space ${\mathcal H}$, and $\omega\subset A$ a finite subset. For $\delta>0$ we put
\[CPA(\pi, A)=\{(\phi,\psi,B)\mid \phi:A\to B, \psi: B\to {\mathcal L}({\mathcal H})\;\;\text{contractive cp maps,}\; dim (B)<\infty\},\]
\[rcp(\pi,\omega,\delta)=\inf\{rank(B)\mid (\phi,\psi,B)\in CPA(\pi, A), ||\psi\circ\phi(a)-\pi(a)||<\delta\;\; \forall\; a\in\omega\},\]
where $rank(B)$ denotes the dimension of a maximal abelian subalgebra of $B$. Since the completely positive rank $rcp(\pi,\omega,\delta)$ is independent of the choice of $\pi$, we may write $rcp(\omega,\delta)$ instead of $rcp(\pi,\omega,\delta)$.

\begin{definition} Let $A\subset {\mathcal L}({\mathcal H})$ be a $C^*$-algebra, $\Phi:A\to A$ a cp map, $\omega\subset A$ finite and $\delta>0$. Put
\[ht(\Phi,\omega,\delta)=\limsup_{k\to\infty}\frac{1}{k}\log(rcp(\bigcup_{i=0}^{k-1}\Phi^i(\omega),\delta)),\;\;ht(\Phi,\omega)=\sup_\delta ht(\Phi,\omega,\delta),\;\; ht(\Phi)=\sup_\omega ht(\Phi,\omega).\]
The number $ht(\Phi)\in [0,\infty]$ is called the topological entropy of $\Phi$. 
\end{definition}
\begin{remark} We collect here some useful facts:

1) If $A_0\subset A$ is a $\Phi$-invariant $C^*$-subalgebra of $A$, then $ht(\Phi\mid_{A_0})\le ht(\Phi)$;

2) If $\{\omega_k\}$ is an increasing sequence of finite subsets of $A$ such that the linear span of $\displaystyle \bigcup _{k,m\ge 0}\Phi^m(\omega_k)$ is dense in $A$, then $\displaystyle ht(\Phi)=\sup_k ht(\Phi,\omega_k)$.

3) If $\sigma:X\to X$ is a continuous map on a compact metric space $X$, then $ht(\tilde{\sigma})=h_{top}(\sigma)$, where $\tilde{\sigma}:C(X)\to C(X),\;\; \tilde{\sigma}(f)=f\circ \sigma$, and $h_{top}$ is the classical topological entropy defined in chapter 7 of \cite{Wa}.
\end{remark}

For the remaining of this paper, we consider a topological graph $E=(E^0,E^1,s, r)$ such that $E^0$ and $E^1$ are compact metric spaces, and $s,r$ are surjective local homeomorphisms. In this case, it follows that the left action of $A=C(E^0)$ on  $H=H(E)$ is by compact operators. Moreover, we assume that the Hilbert module $H=H(E)$ has a basis $\xi_1,\xi_2,...,\xi_n$ satisfying
 \[\langle\xi_i,\xi_j\rangle=\delta_{ij}Q_i,\]
where $Q_i\in A=C(E^0)$ are projections. Suppose that $Q_i=\sum_{j=1}^n\Lambda(i,j)P_j$ for some incidence matrix $\Lambda$, where $P_i=S_iS_i^*$. Recall that the $C^*$-algebra $C^*(E)={\mathcal O}_H$ is generated by $A$ and $n$ partial isometries $S_1, S_2,...,S_n$ with orthogonal ranges satisfying some commutation relations determined by the range map $r$. It contains a copy of the Cuntz-Krieger algebra ${\mathcal O}_\Lambda$. Denote by $G=G(\Lambda)$ the finite graph defined by
the incidence matrix $\Lambda$.
 Notice that $S_\mu$ for $\mu\in \{1,2,...,n\}^k$ is nonzero precisely when $\mu$ is a path of length $k$ in this finite graph. 
We are interested to determine the entropy of the non commutative shifts $\Phi$ and $\Psi$ defined in section 3 for such a topological graph $E$.

\begin{lemma} Consider a topological graph $E$ as above. For each $m\ge 1$ there is a $*$-homomorphism $\chi_m:{\mathcal O}_H\to M_{w(m)}\otimes{\mathcal O}_H$ given by
\[\chi_m(x)=\sum_{|\mu|=|\nu|=m}S_\mu S^*_{\nu}\otimes S^*_\mu xS_\nu,\]
where $w(m)$ denotes the number of paths of length $m$ in the associated graph $G$.
\end{lemma}
\begin{proof}
\[\chi_m(x)\chi_m(y)=(\sum_{\mu,\nu}S_\mu S^*_\nu\otimes S^*_\mu xS_\nu)(\sum_{\mu',\nu'}S_{\mu'}S^*_{\nu'}\otimes S^*_{\mu'}yS_{\nu'})=\]\[=\sum_{\mu,\nu,\mu',\nu'}S_\mu S^*_\nu S_{\mu'}S^*_{\nu'}\otimes S^*_\mu xS_\nu S^*_{\mu'}yS_{\nu'}=\sum_{\mu, \nu'}S_\mu S^*_{\nu'}\otimes S^*_\mu x(\sum_{\nu}S_\nu S^*_\nu)yS_{\nu'}=\chi_m(xy),\]
since $\sum_{|\nu|=m}S_\nu S_\nu^*=1$.
\end{proof}

\begin{theorem} Consider $E=(E^0,E^1, s, r)$ a topological graph with $E^0$ and $E^1$ compact metric spaces such that $s,r$ are surjective local homeomorphisms. Assume that $A=C(E^0)$ is finitely generated, in the sense that there are $a_1,a_2,...,a_q\in A$ such that polynomials in $a_j$ and $a_j^*$ are dense in $A$. Assume also that there are elements $\xi_1, \xi_2,...,\xi_n$ in $H=C(E^1)$ such that
\[\xi=\sum_{i=1}^n\xi_i\langle \xi_i,\xi\rangle \;\;\text {for all}\;\; \xi\in H \; \;\text {and}\;\; \langle \xi_i,\xi_j\rangle=\delta_{ij}Q_i \;\; \text {for some projections}\;\; Q_i\in A,\; i=1,...,n.\]  Let $S_i$ be the image of $\xi_i$ in ${\mathcal O}_H$, and let $P_i=S_iS_i^*$ for $i=1,...,n$. Assume $Q_i=\sum_{j=1}^n\Lambda(i,j)P_j$, for $\Lambda=\Lambda_E$  the incidence matrix of the associated finite graph. Define  the canonical ucp map
\[\Phi: C^*(E)\to C^*(E),\;\; \Phi(c)=\sum_{i=1}^nS_icS_i^*.\]
If $C^*(E)$ is simple, then $ht(\Phi)=\log \rho(\Lambda_E)$, where $\rho(\Lambda_E)$ denotes the spectral radius of the matrix $\Lambda_E$.
\end{theorem}

\begin{proof}
Since $C^*(E)$ contains a copy of ${\mathcal O}_\Lambda$ such that $\Phi({\mathcal O}_\Lambda)\subset {\mathcal O}_\Lambda$, it follows that $ht(\Phi)\ge \log \rho(\Lambda_E)$,  (see also section 6 of \cite{PWY}).
For the other inequality, we use the maps $\chi_m$ from the previous lemma. Since $C^*(E)$ is simple, it follows that $\chi_m:C^*(E)\to \chi_m(C^*(E))\subset M_{w(m)}\otimes C^*(E)$ is a $*$-isomorphism. Note that for $l\ge 1$ we have
\[\Phi^l(c)=\sum_{|\gamma|=l}S_\gamma cS_\gamma^*, \;\; c\in C^*(E).\]
For $k\ge 1$, $|\beta|\le |\alpha|\le k_0$, $m\ge k+k_0$, $l\le k-1$ and $a=a_1^{p_1}\cdots a_q^{p_q},\; 0\le p_j\le k_0, \; j=1,...,q$, one has as in \cite{BG} (see also \cite{SZ})
\[(\chi_m\circ\Phi^l)(S_\alpha aS_\beta^*)=\sum_{j=1}^{h(m,k)}y_j\otimes c_j,\]
where $y_j\in M_{w(m)}$ are partial isometries, and $c_j\in C^*(E)$.

For any $k\ge 1$ consider \[\omega(k)=\{S_\alpha aS_\beta^*:|\beta |\le|\alpha |\le k,\;\; a=a_1^{p_1}\cdots a_q^{p_q},\; 0\le p_j\le k, \; j=1,...,q\}.\] Notice that $\omega_k=\omega(k)\cup\omega(k)^*$ is an increasing sequence of finite subsets of $C^*(E)$ such that the span of their union is dense in $C^*(E)$. For a fixed $k_0\ge 1$ and $\delta>0$ we have
\[\limsup_{k\to\infty}\frac{1}{k}\log rcp(\bigcup_{i=0}^{k-1}\Phi^i(\omega(k_0)), \delta)\le\log \rho(\Lambda_E).\]
Indeed, fix $k\ge 1$, and let $m=k+k_0$. Since $C^*(E)$ is nuclear, there exists $(\phi_0, \psi_0, M_{m_0})\in CPA(id, C^*(E))$ such that 
\[\|\psi_0(\phi_0(c_j))-c_j\|<\frac{\delta}{h(m,k)},\; j=1,...,h. \] 
Consider $B=M_{w(m)}\otimes M_{m_0}$ and let ${\mathcal H}$ be a Hilbert space on which $C^*(E)$ acts faithfully. The $*$-isomorphism $\chi_m^{-1}:\chi_m(C^*(E))\to C^*(E)$ extends to a cp map $\psi_m:M_{w(m)}\otimes C^*(E)\to {\mathcal L}({\mathcal H})$ with $\|\psi_m\|=1$. 
Consider the cp maps $\phi=(id\otimes \phi_0)\circ\chi_m:C^*(E)\to B$ and $\psi=\psi_m\circ(id\otimes\psi_0):B\to{\mathcal L}({\mathcal H})$. For $c=S_\alpha aS_\beta^*\in \omega(k_0)$ we get
\[\|\psi(\phi(\Phi^l(c)))-\Phi^l(c)\|=\|\psi_m(id\otimes \psi_0\phi_0)((\chi_m\circ\Phi^l)(c))-\Phi^l(c)\|=\]\[=\|\psi_m(id\otimes \psi_0\phi_0)((\chi_m\circ\Phi^l)(c))-\psi_m((\chi_m\circ\Phi^l)(c))\|\le \|(id\otimes\psi_0\phi_0)((\chi_m\circ\Phi^l)(c))-(\chi_m\circ\Phi^l)(c)\|=\]\[
=\|\sum_{j=1}^{h(m,k)}y_j\otimes(\psi_0\phi_0(c_j)-c_j)\|<\delta.\]
Hence \[rcp(\bigcup_{i=0}^{k-1}\Phi^i(\omega(k_0)), \delta)\le m_0w(m)=m_0w(k+k_0),\]
which gives
\[\limsup_{k\to\infty}\frac{1}{k}\log(\bigcup_{i=0}^{k-1}\Phi^i(\omega(k_0)), \delta)\le \limsup_k\frac{1}{k}\log w(k)=\log\rho(\Lambda_E).\] 
We conclude that $ht(\Phi)=\log \rho(\Lambda_E)$.
\end{proof}
\begin{proposition}
Consider $E=(E^0,E^1, s, r)$ a topological graph as in the previous theorem.  Assume that the shift $\sigma_E:E^\infty\to E^\infty$ is expansive. Then the non commutative shift $\Psi :{\mathcal F}_E\to {\mathcal F}_E$ defined by the maps 
$\Psi_k: C^*(R_k)\to C^*(R_{k+1})$ has entropy $ht(\Psi)\ge h_\ell(E)$.
\end{proposition}
\begin{proof}
The natural diagonal $C(E^\infty)$ of the groupoid $\Gamma(E)$ defined at the end of section 3 is invariant under the map $\Psi$ and its restriction $\Psi\mid_{C(E^\infty)}$ is conjugated to the map induced by shift $\sigma_E:E^\infty\to E^\infty$, hence $ht(\Psi)\ge h_{top}(\sigma)$. Now apply Corollary 4.1.
\end{proof}
\begin{corollary} 
 The non commutative shifts $\Phi, \Psi :{\mathcal F}_E\to {\mathcal F}_E$  may have different entropies.
\end{corollary}
\begin{proof}
Since ${\mathcal F}_E$ is invariant under $\Phi$, from Theorem 5.1 we get  $ht(\Phi\mid_{{\mathcal F}_E})=\log\rho(\Lambda_E)$. From the above proposition, $ht(\Psi)\ge h_\ell(E)$.
 Examples 6.2 and 6.3 in the next section show that $ht(\Phi)$ and $ht(\Psi)$  indeed may be different.

\end{proof}

\section{Examples}

\example For $X$ be a compact space, and $\sigma:X\to X$ a continuous surjective map we defined earlier the topological graph $E=E(X,\sigma)=(E^0,E^1,s,r)$, where $E^0=E^1=X,\; s=id$ and $r=\sigma$. 
We identify the Hilbert bimodule $H=H(E)$ and its tensor powers $H^{\otimes n}$ with $C(X)$. The endomorphism $\tilde{\sigma}:C(X)\to C(X),\; \tilde{\sigma}(f)=f\circ\sigma$ coincides with the left action $\varphi_r:C(X)\to {\mathcal K}(H)$ defined by $r=\sigma$, after the identification of ${\mathcal K}(H)$ with $C(X)$. Consider $(\tau, \pi)$ a Toeplitz representation of $H$. Then the constant function $1$ in $C(X)$ determines an element $U$ in $C^*(\tau, \pi)$ satisfying $U^*U=1$,
$U\pi(f)=\tau(f)$ and $\pi(f)U=\tau(\tilde{\sigma}(f))$ for $f\in C(X)$. The map $\psi:{\mathcal K}(H)\to C^*(\tau, \pi)$ can be expressed as $\psi(f)=U\pi(f)U^*$, hence $\psi(\varphi_r(f))=U\pi(\tilde{\sigma}(f))U^*$. The $C^*$-algebra ${\mathcal O}_E$ is the universal $C^*$-algebra generated by a copy of $C(X)$ and a unitary $u$ satisfying $u^*fu=\tilde{\sigma}(f)$. In particular, for $\sigma$ a homeomorphism, ${\mathcal O}_E$ is isomorphic to the crossed product $C(X)\rtimes_{\sigma}{\mathbb Z}$. For details, see the results following Example 2 in \cite{Ka1}.

The Hilbert bimodule $H$ has basis $\{\xi_1\}$, where $\xi_1$ is determined by the constant function $1$. The canonical map $\Phi:{\mathcal O}_E\to {\mathcal O}_E$ is given by $\Phi(c)=ucu^*$. Note that $\Phi(\tilde{\sigma}(f))=f$, so $\Phi$ is a left inverse for $\tilde{\sigma}$. If $\sigma$ is a homeomorphism, then it is   known that $ht(\Phi)= h_{top}(\sigma)$.

\example Consider again the topological graphs $E_{p,q}$ of Katsura  \cite{Ka3}, where  $p, q$ are integers with $p\ge 1$ and $q\neq 0$, $E^0=E^1={\mathbb T}, \;s(z)=z^p$ and $r(z)=z^q$, as in Example 4.2. The Hilbert bimodule $H_{p,q}=H(E_{p,q})=C(E^1)$ has a basis $\{\xi_1, \xi_2,...,\xi_p\}$, where $\xi_k(z)=\frac{1}{\sqrt{p}}z^{k-1}, k=1,...,p$. We have
\[\langle \xi_j,\xi_k\rangle(z)=\frac{1}{p}\sum_{w^p=z}\overline{w^j}w^k=\delta_{jk}\cdot 1\; \;\text{and}
\;\;\sum_{j=1}^p\theta_{\xi_j,\xi_j}=1.\]Indeed, let $z=exp(it)$. Then $w^p=z$ has solutions $w_m=exp(i(t+2m\pi)/p), m=0,...,p-1$. We have \[\sum_{w^p=z}\overline{w^j}w^k=\sum_{m=0}^{p-1}exp(i(k-j)(t+2m\pi)/p)=exp(i(k-j)t/p)\sum_{m=0}^{p-1}(exp(i(k-j)2\pi/p))^m=\delta_{kj}\cdot 1.\] A similar computation gives
\[\sum_{j=1}^p\xi_j\langle \xi_j,\eta\rangle(z)=\sum_{j=1}^p\xi_j(z)\sum_{w^p=z^p}\overline{\xi_j(w)}\eta(w)=\eta(z).\]
It follows that the $C^*$-algebra ${\mathcal O}_{E_{p,q}}$ is generated by the unitary $u\in C(E^0)=C({\mathbb T})$, where $u(z)=z$  and $p$ isometries $S_1, S_2,...,S_p$ satisfying the relations
\[ \;S_j^*S_k=\delta_{jk}\cdot 1,\;\; \sum_{j=1}^pS_jS_j^*=1,\;\;
 uS_k=S_{k+q},\]
where $S_{k+q}=S_lu^m$ using the unique $l\in\{1,2,...,p\}$ with $k+q=l+pm$.
The isometries $S_1,S_2,...,S_p$ generate a copy of the Cuntz algebra ${\mathcal O}_p$, and the finite graph $\Lambda$ associated to the topological graph $E_{p,q}$ has one vertex and $p$ edges.
From Theorem 5.1, we get \[ht(\Phi)=\log p.\]
We have ${\mathcal L}(H_{p,q})={\mathcal K}(H_{p,q})\cong {\mathbb M}_p\otimes C({\mathbb T})$. Let $q=pm+r$, where $r\in\{0,1,...,p-1\}$. If $1\le r+k\le p$, then 
$u\cdot \xi_k=\xi_{r+k}\cdot u^m,$
and therefore the inclusion $\Psi_0=\varphi: C(E^0)\to {\mathcal K}(H_{p,q})$ is given by
\[u\mapsto \sum_{k=1}^pS_{r+k}u^mS_k^*;\]
 if  $ p+1\le r+k\le 2p-1$, then 
$u\cdot\xi_k=\xi_{r+k-p}\cdot u^{m+1},$
 and $\Psi_0$ it is given by
\[u\mapsto \sum_{k=1}^pS_{r+k-p}u^{m+1}S_k^*.\]
The $C^*$-algebra ${\mathcal F}_{E_{p,q}}$ is isomorphic to the inductive limit
\[C({\mathbb T})\stackrel{\Psi_0}{\to} {\mathbb M}_p\otimes C({\mathbb T})\stackrel{\Psi_1}{\to}{\mathbb M}_{p^2}\otimes C({\mathbb T})\stackrel{\Psi_2}{\to} ...,\]
where the embeddings $\Psi_k$ are obtained from $\Psi_0$ by tensoring with the identity.
Recall from Example 4.2 that the space of infinite paths $E_{p,q}^{\infty}$ is homeomorphic to 
\[\Sigma(p,q)=\{(z_0,z_1,z_2,...)\in {\mathbb T}^{\mathbb N}\mid z_k^p=z_{k+1}^q\},\] 
which is a $1$-dimensional solenoid if $(p,q)=1$.
The  map $\Psi:{\mathcal F}_{E_{p,q}}\to {\mathcal F}_{E_{p,q}}$ determined by the maps $\Psi_k$, when restricted to $C(E_{p,q}^{\infty})$,   is conjugated with
the shift $\sigma$, which has entropy $h_{top}(\sigma)\ge\max\{\log p, \log|q|\}$ for $(p,q)=1$.
It follows that $ht(\Psi)\ge\log \max\{p, |q|\}$.

\example Consider the topological graph $E=(E^0,E^1,s,r)$ from Example 4.3, where $E^0={\mathbb T}, E^1={\mathbb T}\times\{1,2\}, s(z,1)=z^2, r(z,1)=z, s(z,2)=z, r(z,2)=z^3$. 
The Hilbert bimodule $H=H(E)=C(E^1)$ has a basis $\{\xi_1, \xi_2, \xi_3\}$, where $\xi_k(z,1)=\frac{1}{\sqrt{2}}z^{k-1}, \xi_k(z,2)=0$ for $ k=1,2$ and $\xi_3(z,1)=0,\; \xi_3(z,2)=1$. We have $\langle \xi_j,\xi_k\rangle=\delta_{jk}\cdot 1$. The $C^*$-algebra ${\mathcal O}_E$ is generated by $C(E^0)=C({\mathbb T})$ and three isometries $S_1,S_2,S_3$ satisfying the relations
\[\sum_{j=1}^3S_jS_j^*=1, \;\;S_j^*S_k=\delta_{jk}\cdot 1,\;\;
f\cdot S_j=\sum_{k=1}^3S_k\langle \xi_k,f\cdot \xi_j\rangle,\]
where $(f\cdot \xi_j)(z,1)=f(z)\frac{1}{\sqrt{2}}z^{j-1}, j=1,2, (f\cdot \xi_3)(z,2)=f(z^3)$ for $f\in C({\mathbb T})$. It follows from Theorem 5.1 that the $C^*$-algebra ${\mathcal O}_E$ contains a copy of the Cuntz algebra ${\mathcal O}_3$ and $ht(\Phi)=\log 3$.
Let $u\in C({\mathbb T})$ be the generator  $u(z)=z$. We get the  relations
\[uS_1=S_2,\;\; uS_2=S_1u,\;\; uS_3=S_3u^3.\]
We have ${\mathcal L}(H)={\mathcal K}(H)\cong {\mathbb M}_3\otimes C({\mathbb T})$. The inclusion $\Psi_0: C(E^0)\to {\mathcal L}(H)$ is determined by the map \[u\mapsto \left(\begin{array}{ccc}0&u&0\\1& 0&0\\0&0&u^3\end{array}\right).\]
 Indeed, the element $u\in C(E^0)$ is sent into the compact operator $\xi_2\otimes \xi_1^*+\xi_1u\otimes \xi_2^*+\xi_3u^3\otimes \xi_3^*$, because
\[(u\cdot\xi_1)(z,1)=z\xi_1(z,1)=\frac{1}{\sqrt{2}}z=\xi_2(z,1),\]
\[(u\cdot\xi_2)(z,1)=z\xi_2(z,1)=\frac{1}{\sqrt{2}}z^2=(\xi_1\cdot u)(z,1),\]
\[(u\cdot\xi_3)(z,2)=z^3\xi_3(z,2)=z^3=(\xi_3\cdot u^3)(z,2).\]
The $C^*$-algebra ${\mathcal F}_E$ is isomorphic to the inductive limit
\[C({\mathbb T})\to {\mathbb M}_3\otimes C({\mathbb T})\to{\mathbb M}_{3^2}\otimes C({\mathbb T})\to ...\]
The first inclusion is given by $u\mapsto S_2S_1^*+S_1uS_2^*+S_3u^3S_3^*$, and the other inclusions are obtained by iterating this formula.
The restriction of  $\Psi:{\mathcal F}_E\to{\mathcal F}_E$ to the diagonal $C(E^\infty)$ is conjugated to the map induced by the shift $\sigma$ on the generalized solenoid 
\[E^\infty=\{(z_k,a_k)_{k\ge 0}\in ({\mathbb T}\times \{1,2\})^{\mathbb N}\mid a_k=1 \Rightarrow \; z_{k+1}^2=z_k, \;a_k=2 \Rightarrow \;z_{k+1}=z_k^3\},\]
and has entropy $ht(\Psi)\ge \log 4$, see Example 4.3.

\example Consider the topological graph from Example 4.4. 
In the case $E^0$ and $E^1$ are finite, the Hilbert bimodule 
$H(p,q)=H(E\times_{p,q}{\mathbb T})$ has a basis $\{\xi_{e,k}\mid e\in E^1, k=1,...,p(e)\}$, where $\xi_{e,k}(e,z)=\frac{1}{\sqrt{p(e)}}z^{k-1}$ and $\xi_{e,k}(e',z)=0$ for $e'\neq e$. For each $v\in E^0$, let $ u_v\in {\mathcal O}_{E\times_{p,q}{\mathbb T}}$ be the image of  the generating unitary of $C(\{v\}\times {\mathbb T})\subset C(E^0\times {\mathbb T})$. Let $S_{e,k}\in {\mathcal O}_{E\times_{p,q}{\mathbb T}}$  be the image of $\xi_{e,k}$. Then ${\mathcal O}_{E\times_{p,q}{\mathbb T}}$ is generated by the family $\{u_v\}_{v\in E^0}$ of partial unitaries with orthogonal ranges and the family $\{S_{e,k}\}_{e\in E^1,1\le k\le p(e)}$ of partial isometries with orthogonal ranges satisfying the relations
\[S^*_{e,k}S_{e,k}=u^*_{s(e)}u_{s(e)}\; \text{for}\;e\in E^1\;\text{and}\;1\le k\le p(e),\]
\[u_{r(e)}S_{e,k}=S_{e,k+q(e)}\; \text{for}\;e\in E^1\;\text{and}\;1\le k\le p(e),\]
\[u^*_vu_v=\sum_{r(e)=v}S_{e,k}S^*_{e,k}\;\text{for}\; v\in E^0,\]
where $S_{e,k+q(e)}=S_{e,k'}u^l_{s(e)}$ for the unique $k'\in\{1,2,...,p(e)\}$ and $l\in{\mathbb Z}$ with $k+q(e)=k'+p(e)l$, see Appendix A in \cite{Ka3}. 
Consider the matrices $P,Q$, where
\[P(v,w)=\sum_{e\in s^{-1}(v)\cap r^{-1}(w)}p(e), \;\;Q(v,w)=\sum_{e\in s^{-1}(v)\cap r^{-1}(w)}q(e)\] for $v,w\in E^0$. The partial isometries $S_{e,k}$ generate a copy of the Cuntz-Krieger algebra ${\mathcal O}_\Lambda$ inside ${\mathcal O}_{E\times_{p,q}{\mathbb T}}$, and by Theorem 5.1 we get
\[ht(\Phi)=\log\rho(\Lambda).\]
If the formulas conjectured in Example 4.4 were true, this would imply that, under certain conditions on the maps $p$ and $q$, 
\[ht(\Psi)\ge\log\max\{\rho(P),\rho(Q)\}.\]

\end{document}